\documentclass[a4paper]{article}
\usepackage{amsmath, amsthm, amsfonts, amssymb, mathrsfs}
\usepackage[all]{xy}
\usepackage{color}
\usepackage{afterpage}
\usepackage{xcolor}
\usepackage[backref=page]{hyperref}
\usepackage[capitalise]{cleveref}

\theoremstyle{plain}
\newtheorem{theorem}{Theorem}[section]
\newtheorem{proposition}[theorem]{Proposition}
\newtheorem{corollary}[theorem]{Corollary}
\newtheorem{lemma}[theorem]{Lemma}
\newtheorem{remark}[theorem]{Remark}
\newtheorem{definition}[theorem]{Definition}

\newtheorem{notation}[theorem]{Notation}

\newtheorem*{theorem*}{Theorem}
\newtheorem*{proposition*}{Proposition}
\newtheorem*{corollary*}{Corollary}
\newtheorem*{lemma*}{Lemma}
\theoremstyle{definition}
\newtheorem*{definition*}{Definition}
\newtheorem{theoremintro}{Theorem}
\newtheorem*{notation*}{Notation}

\newcommand{\CC}{\mathbb{C}}
\newcommand{\PP}{\mathbb{P}}
\newcommand{\ZZ}{\mathbb{Z}}

\newcommand{\NN}{\mathbb{N}}
\newcommand{\OO}{\mathcal{O}}

\newcommand{\FF}{\mathcal{F}}
\newcommand{\EE}{\mathcal{E}}
\newcommand{\LL}{\mathcal{L}}

\newcommand{\KK}{\mathcal{K}}
\newcommand{\HH}{\mathcal{H}}

\newcommand{\II}{\mathscr{I}}
\newcommand{\JJ}{\mathscr{J}}
\newcommand{\pp}{\mathfrak{p}}
\DeclareMathAlphabet{\mathpzc}{OT1}{pzc}{m}{it}

\newcommand{\codim}{\mathrm{codim}}
\newcommand{\ann}{\mathrm{ann}}

\newcommand{\sing}{\mathrm{Sing}}
\newcommand{\per}{\mathpzc{Per}(\omega)}
\newcommand{\kup}{\mathpzc{Kup}(\omega)}

\newcommand{\wt}{\widetilde}

\title{The singular locus of $q$-logarithmic foliations}
\author{Ariel Molinuevo \thanks{The author was fully supported by Universidade Federal do Rio de Janeiro, Brazil.} and Federico Quallbrunn \thanks{The author was fully supported by CONICET, Argentina.}}

\newenvironment{acknowledgement}
{\section*{Acknowledgments}}
{}

\begin{document}
\maketitle

\begin{abstract}
We determine the structure of the singular locus of generic \linebreak codimension-$q$ logarithmic foliations and its relation with the unfoldings of these foliations. We calculate the ideal defining the scheme of persistent singularities.
\end{abstract}

\

{\footnotesize
\noindent
Keywords: $q$-logarithmic foliations, projective varieties, singular locus.
}

\

{\footnotesize
\noindent
Mathematics Subject Classification (2020): 14A15, 32M25.
}


\

\section{Introduction}

The study of holomorphic foliations on algebraic varieties has been an active area of research in mathematics for several decades. We refer to the book \cite{scardua2021} for a review of fundamental results and open problems in the subject. Singularities of foliations are of particular interest as they provide valuable insight into the dynamics and geometry of the foliation and how these change under deformation.

Starting from the seminal work of Calvo-Andrade in \cite{omegar} until recent years, there has been significant progress in the study of foliations defined by logarithmic forms. In \cite{omegar}, Calvo-Andrade proves that the set of codimension $1$ foliations defined by logarithmic $1$-forms in a compact holomorphic variety $X$ with $H^1(X,\CC)=0$ is stable under deformation. In the case where $X=\PP^n_\CC$ this result implies that foliations defined by forms of the type
\[
	\omega= \left(\prod_{j=1}^r F_j\right) \sum_{i=1}^r \lambda_i\frac{dF_i}{F_i},\quad \lambda_j\in\CC,\ F_j\in \CC[x_0,\dots,x_n],
\]
 $F_j$ homogeneous, $\deg F_j=d_j$, define an irreducible component $\mathcal{L}(d_1,\dots,d_r)$ of the moduli space of foliations of codimension $1$ in $\PP^n$.
Recently, in \cite{fjc}, a new proof of this result is given using deformation theory.
In recent years, some progress has also been made in the classification of codimension $q$ foliations. In particular, Cerveau and Lins-Neto in \cite{celnlog} have proven that foliations defined by logarithmic $q$-forms of the type
\[
	\omega= \left(\prod_{j=1}^r F_j\right) \left[\sum_{\substack{I\subset \{1,\dots,r\}  \\ |I|=q}} \lambda_I \left(\bigwedge_{i\in I}\frac{dF_i}{F_i}\right)\right],\quad \lambda_I\in\CC,\ F_j\in\CC[x_0,\dots, x_n],
\]
define irreducible components of the moduli space of codimension $q$ foliations on $\PP^n$,  and Gargiulo in \cite{javi2020} gives an independent proof of the case $q=2$ using deformation theory.

\bigskip

This article deals with foliations defined by logarithmic forms such as the ones above. Instead of focusing on the stability of such forms under deformation, we deal with their singular locus and their unfoldings, generalizing the main result of
 \cite{fmi} to the case of arbitrary codimension, and applying this result to the characterization of the unfoldings of these foliations, as we will explain now.

\bigskip

Unfoldings of singular foliations, and particularly \emph{first order unfoldings} (see \cref{defUnfolding} below), were introduced independently by Suwa and Mattei (see \cite{suwa-unfoldings_complex, mattei91} and the references within). Although the original definition of Suwa was stated in full generality, the focus of these earlier articles was on germs of codimension-1 foliations around a singularity.

In later years, the work of Suwa has been expanded in several articles, see \cite{mmq,mmq2} and the references within. In this paper we will apply the results from \cite{mmq2} to study the singular locus of a $q$-logarithmic foliation. That paper is a generalization of \cite{mmq}. Since we will refer to both articles several times, we will briefly explain what they are about here.

In \cite{mmq} we studied first order unfoldings and we defined the \emph{Kupka scheme} of a codimension-1 foliation on $\PP^n$.

A codimension-1 foliation in $\PP^n$ can be given by a global twisted differential form $\omega\in H^0(\PP^n,\Omega^1_{\PP^n}(e))$, for some $e\geq 2$ defining a line bundle, such that $\omega$ verifies the \emph{Frobenius integrability condition}:
\[
 \omega\wedge d\omega = 0\ ,
\]
see \cref{folcodim1}.

A first order unfolding of a foliation defined by $\omega$ is a particular type of first order deformation of $\omega$. While a (first order) deformation of $\omega$ is given by a family of differential forms $\omega_\varepsilon$ parameterized by an infinitesimal parameter $\varepsilon$ such that $\omega_\varepsilon$ is integrable for every fixed value of $\varepsilon$ and coincides with $\omega$ in the origin, a (first order) unfolding $\widetilde{\omega}_\varepsilon$ is given by a foliation in an infinitesimal neighbourhood of $\PP^n$ such that it restricts to $\omega$ in the central fiber.

\

In the codimension $1$ case, first order unfoldings define a sheaf isomorphic to a sheaf of ideals, which is called the \emph{unfolding ideal}, $I(\omega)$, see \cref{remI}. If we define as $J(\omega)$ the ideal defined by the coefficients of $\omega$, we have that the variety defined by $J(\omega)$ is the singular locus of $\omega$, $\sing(\omega)$, see \cref{singw}. It can be also easily seen that $J(\omega)\subset I(\omega)$. So, the first order unfoldings define a subvariety of $\sing(\omega)$. We refer to \cite{moli} for details on this topic.

In the literature, see \cite{kupka}, \cite[Fundamental Lemma, p.~406]{deMedTesis} or \cite[Definition 2.1]{gmln}, the \emph{Kupka set} of a codimension-1 foliation $\omega$ is the set of points such that $\omega(p)=0$ and $d\omega(p)\neq 0$. To work in an algebraic setting first we needed to associate a scheme structure to that set. The way to do that was to define the Kupka scheme as the schematic support of $d\omega$ inside $\sing(\omega)$.

For a more formal definition, let us fix here $S= \CC[x_0,\ldots,x_n]$ the homogeneous coordinate ring of $\PP^n$. We defined the Kupka scheme, $\kup$, see \cite[Definition 4.1]{mmq}, as $\kup=Proj(S/K(\omega))$ where $K(\omega)$ is the homogeneous ideal defined as
\[
K(\omega) = ann(\overline{d\omega})+J(\omega), \ \ \overline{d\omega}\in \Omega^2_S\otimes S/J(\omega)\ .
\]

We also stated the following relation between the ideal of first order unfoldings and the Kupka ideal, see \cite[Theorem 4.12]{mmq} for a complete statement, for $\omega$ a \emph{generic} codimension-1 foliation in $\PP^n$ we have:
\begin{equation}\label{teo0.1}
 \sqrt{I(\omega)} = \sqrt{K(\omega)}\ .
\end{equation}

\

Despite the fact that a first order unfolding of a codimension-1 foliation identifies with an ideal, in higher codimension there was no analog to that fact. In \cite{mmq2} we defined what we called the \emph{subscheme of persistent singularities}, $\per$, addressing that particular situation, where $\omega$ defines a codimension-$q$ foliation, see \cref{codq} for a proper definition. We refer to \cref{defII} for the formal definition of $\per$, but for now we can state an equivalent definition, see \cref{prop}, as:
\[
 \pp\in\per\iff \pp\text{ is a singular point of every unfolding of }\omega\ .
\]

We also gave a definition of the Kupka scheme for a codimension-$q$ foliation, see \cref{defKK}.

\

In \cite[Theorem 3]{fmi} the authors show that the singular locus of a generic codimension-1 logarithmic foliation in $\PP^n$ is given by a codimension 2 variety (which is the Kupka variety) and a set of isolated points. In this paper we extend that result showing that the singular locus of a \emph{generic} codimension-$q$ logarithmic foliation, see \cref{genericity} for the generic conditions, in a smooth variety $X$ is given by two varieties $\kup$ and $\HH$ where $\kup$ has codimension $(q+1)$ and $\HH$ has dimension $(q-1)$ or $\HH=\emptyset$, see \cref{teo2}:
\begin{theoremintro}
Let $\omega\in H^0(X,\Omega^q_X\otimes \LL)$ be a generic, locally decomposable, $q$-logarithmic foliation. Then we can decompose its singular locus $\sing(\omega)$ as the disjoint union of
\[
 \sing(\omega) = \kup\sqcup \HH,
\]
where $\kup$ is the Kupka variety of $\varpi$ of codim $q+1$ and $\HH$ is a variety of dimension $q-1$ or $\HH=\emptyset$. Even more so, we have that $\kup$ coincides with the scheme of persistent singularities $\per$.
\end{theoremintro}

\

This article is organized as follows:

\

In \cref{codq} we give the definitions of first order unfolding, of persistent singularities and of the Kupka scheme all of them for a codimension-$q$ foliation on a smooth variety $X$.

\

In \cref{poles} we recall from \cite{deligne} the definition of logarithmic differential form and of the residue of such form.

\

Finally in \cref{main} we define $q$-logarithmic foliations and we use the relation between Kupka singularities and persistent singularities to establish in \cref{teo2} the dimension of the irreducible components of the singular set of a generic logarithmic foliation on $X$, as well as determining the equality between the schemes of Kupka singularities and persistent singularities.

Then in \cref{corofinal} we use the result stated above to compute the ideal of persistent singularities.

\

These results are of particular interest as not much is known in general about the singular set of a higher codimension foliation, so logarithmic foliations represent an interesting first example where the dimension of the irreducible components of the singular set as well as the scheme of persistent singularities are defined.

\

\begin{acknowledgement}
We would like express our gratitude to Alicia Dickenstein and Jorge Vitorio Pereira for their valuable suggestions and contributions.
\end{acknowledgement}

\section{Codimension-$q$ foliations}\label{codq}

Along this section we give the definition of \emph{codimension-$q$ foliation} on a smooth variety $X$. Then we recall the definition of unfolding of a codimension-$q$ foliation on a variety $X$. Finally, we review the definitions of persistent singularities and of Kupka singularities for codimension-$q$ foliations, see Definition \ref{defII} and Definition \ref{defKK}, respectively, and we state some general facts regarding the ideal of persistent singularities and the Kupka scheme.
We also add remarks recalling the analogous definitions for a codimension-1 foliation in $\PP^n$ for the sake of clarity.

We refer the reader to \cite{mmq2} for a complete treatment of the subject.

\bigskip

Let us fix a smooth variety as $X$. If $\Xi \in \Gamma(U, \bigwedge^p TX)$ is a multivector and $\varpi\in \Gamma(U,\Omega^q_X)$ a $q$-form we will denote by $i_\Xi \varpi \in \Gamma(U,\Omega^{q-p}_X)$ the contraction.
Recall from \cite[Proposition 1.2.1]{deMed1} that the \emph{Pl\"ucker relations} for $\varpi$ are given by
\[
 i_\Xi \varpi\wedge \varpi=0
\]
for any $\Xi\in\bigwedge^{q-1} TX$.
This is equivalent to: if $\varpi(\pp)\neq 0$ for some closed point $\pp\in X$ then $\varpi$ is \emph{locally decomposable} as a product $\varpi=\varpi_1\wedge\dots\wedge\varpi_q$ of $q$ 1-forms at a neighborhood of $\pp$.  We will call a differential form $\varpi$ $\emph{locally decomposable}$ if it is locally decomposable at every closed point $\pp\in X$.

\begin{definition}\label{folcodimq}
Let $\LL$ be a line bundle and $\omega:\LL\to \Omega^q_X$, with $1\leq q\leq dim(X)-1$, be a (non trivial) morphism of sheaves, we will say that the morphism is \emph{integrable} if
\begin{itemize}
\item[a)] For every local section $s$ of $\LL$ and $\Xi$ of $\bigwedge^{q-1} TX$,
$\omega(s)$ verifies
\begin{equation}\label{frobenius2}
 d (i_\Xi \omega(s))\wedge \omega(s)=0.
\end{equation}
\item[b)] $\Omega^q_X/\LL$ is torsion free.
\item[c)] The map
\[
 i_\Xi \omega\wedge \omega:\LL\to\Omega^{q+1}_X\otimes\LL^{-1}
\]
is zero for every local section $\Xi$ of $\bigwedge^{q-1} TX$.

\end{itemize}

We also say that $\omega$ determines a \emph{codimension-$q$ foliation}.
\end{definition}

\bigskip

Let $\omega:\LL\to\Omega^q$ be a integrable morphism. Then, we can consider
two maps,
\[
\xymatrix{
\bigwedge^{q-1} TX\otimes \LL\ar[r]^<<<<<{i_{(-)}\omega}
&\Omega^1_X\ar[r]^<<<<<{\omega\wedge -}&\Omega_X^{q+1}\otimes\LL^{-1}
}
\]
Condition c) on $\omega$ implies that this diagram is a complex
and it is easy to check that its homology is supported over the points
where $\omega$ is not decomposable.

\begin{remark}\label{folcodim1}
In the case of the projective space $\PP^n$ a codimension-1 foliation can be given by a global twisted differential form $\omega\in H^0(\PP^n,\Omega^1_{\PP^n}(e))$ such that it verifies the Frobenius integrability condition
\[
 \omega\wedge d\omega=0\ ,
\]
see \cite{fmi,moli,mmq}. Or, more similar to \cref{folcodimq}, by a morphism $\omega:\OO_{\PP^n}(-e)\to \Omega^1_{\PP^n}$.

The integrability condition is equivalent to item a) of \cref{folcodimq}. The item b) of the definition is equivalent to ask the generic property of having singular locus of codimension $\geq 2$. There is no equivalent for the item c) in the case of codim-1.

Let us denote the ring of homogeneous coordinates of $\PP^n$ by $S=\CC[x_0,\ldots,x_n]$ and the homogeneous component of degree $\ell$ of a graded $S$-module $M$ by $M(\ell)$. A global form, $\omega\in H^0(\PP^n,\Omega^1_{\PP^n}(e))$, can be written as
\[
 \omega = \sum_{i=0}^n A_i dx_i \ ,
\]
where the $A_i\in S(e-1)$ are homogeneous polynomials of degree $e-1$. The condition of descent of $\omega$ to $\PP^n$ is given by the annihilation of the contraction with the radial field $R=\sum_{i=0}^n x_i \frac{\partial}{\partial x_i}$, \emph{i.e.},
\[
 i_R(\omega) = \sum_{i=0}^n A_i x_i = 0\ .
\]

\end{remark}

\begin{definition}\label{def1}
We define the \emph{conormal sheaf associated with the foliation $\omega$}, denoted $\EE=\EE(\omega)$, as the kernel of $\omega\wedge -:\Omega^1_X\to \LL^{-1}\otimes\Omega^{q+1}_X$.
\end{definition}

Notice that by \cite[Proposition 1.1]{hartstable} we get that $\EE$ is a reflexive sheaf.

\bigskip

Composing a morphism $\omega:\LL\to \Omega^q$ with the contraction of forms with vector fields gives us a morphism
\[
\bigwedge^q TX\otimes \LL\to\OO_X.
\]
\begin{definition}\label{def-sing}
The ideal sheaf $\JJ(\omega)$ is defined to be the sheaf-theoretic image of the morphism $\bigwedge^q TX\otimes \LL\to\OO_X$. The subscheme it defines is called \emph{the singular scheme} of $\omega$ and denoted $\sing(\omega)\subseteq X$. We will denote it just by $\JJ$ if no confusion arises.
\end{definition}

\begin{remark}\label{singw}
 In the case of a codimension-1 foliation in the projective space $\PP^n$ we said in \cref{folcodim1} that $\omega$ can be written as $\omega = \sum_{i=0}^n A_i dx_i$ for some homogeneous polynomials $A_i\in S(e-1)$. In this situation the definition above coincides with the ideal $J(\omega) = (A_0,\dots, A_n)$ and the variety defined by it.
\end{remark}

From \cite{suwa-unfoldings_complex}[(4.6) Definition] we get the following definition for a codimension-$q$ foliation:

\begin{definition}\label{defUnfolding}
Let $T$ be a scheme, $\pp\in T$ a closed point, and $\LL\xrightarrow{\omega} \Omega^q_X$ a codimension-$q$ foliation on $X$. An \emph{unfolding} of $\omega$ is a codimension-$q$ foliation $\widetilde{\LL}\xrightarrow{\widetilde{\omega}}\Omega^q_{X\times T}$ on $X\times T$ such that $\widetilde{\omega}|_{X\times\{\pp\}}\cong \omega$. In the case $T=\mathrm{Spec}(\CC[\varepsilon]/(\varepsilon^2))$ we will call $\widetilde{\omega}$ a \emph{first order unfolding}.
\end{definition}

Let $\mathcal{D}=\mathrm{Spec}(\CC[\varepsilon]/(\varepsilon^2))$ be the scheme of \emph{dual numbers} over $\CC$, $0\in \mathcal{D}$ be its closed point, $\pi:X\times \mathcal{D}\to \mathcal{D}$ be the projection and $\iota:X\cong X\times\{0\}\hookrightarrow X\times \mathcal{D}$ be the inclusion.
Then the sheaf $\Omega^q_{X\times \mathcal{D}}$ can be decomposed as a direct sum of $\iota_*(\OO_X)$-modules as
 \[
  \Omega^q_{X\times \mathcal{D}}\cong \iota_*\Omega^q_X\oplus \varepsilon\cdot (\iota_*\Omega^q_X)\oplus \iota_*\Omega^{q-1}_X\wedge d\varepsilon.
 \]
Given a codimension-$q$ foliation determined by an integrable morphism $\LL\xrightarrow{\omega}\Omega^q_X$, and a first order unfolding $\widetilde{\omega}: \widetilde{\LL}\to \Omega^q_{X\times \mathcal{D}}$ of $\omega$, we take local generators $\varpi$ of $\LL(U)$ and $\widetilde{\varpi}$ of $\widetilde{\LL}(U\times \mathcal{D})$. Suppose $\omega$ and $\widetilde{\omega}$ are locally decomposable, then we may take $U$ small enough such that $\varpi$ and $\widetilde{\varpi}$ decompose as products
\[
 \varpi=\varpi_1\wedge\dots\wedge\varpi_q, \qquad \widetilde{\varpi}=\widetilde{\varpi}_1\wedge\dots\wedge \widetilde{\varpi}_q.
\]
Then we can write $\widetilde{\varpi}_i=\varpi_i + \varepsilon \eta_i + h_i d\varepsilon$
and the equations $d\widetilde{\varpi}_i\wedge \widetilde{\varpi}=0$ for $i=1,\dots, q$
are equivalent to the equations

\begin{equation}\label{ecu0.2}
 \left\{
\begin{aligned}
&d\eta_i\wedge \varpi+ d\varpi_i\wedge \left(\sum_{j=1}^q (-1)^j\eta_j \wedge\varpi_{\widehat{j}}\right)=0,\qquad (i=1,\dots,q), \\
&(dh_i-\eta_i)\wedge \varpi+ d\varpi_i\wedge \left(\sum_{j=1}^q (-1)^jh_j \varpi_{\widehat{j}}\right)=0,\qquad (i=1,\dots,q),
\end{aligned}
\right.
\end{equation}

where $ \varpi_{\widehat{j}}=\varpi_1\wedge\dots\wedge\varpi_{j-1}\wedge\varpi_{j+1}\wedge\dots\wedge\varpi_q\in \Omega^{q-1}_X(U)$.

\bigskip

As is shown in \cite[proof of 6.1 Theorem]{suwa-unfoldings_complex} the second equation implies the first. So we finally get that the equations $d\widetilde{\varpi}_i\wedge \widetilde{\varpi}=0$ for $i=1,\dots, q$ are equivalent to

\begin{equation}\label{equnfcodq}
\begin{aligned}
 \left\{
 (dh_i-\eta_i)\wedge \varpi+ d\varpi_i\wedge \left(\sum_{j=1}^q (-1)^jh_j \varpi_{\widehat{j}}\right)=0,\qquad (i=1,\dots,q)\ .
\right.
\end{aligned}
\end{equation}

\begin{remark}\label{remI}
 In the case of a codimension-1 foliation in the projective space $\PP^n$ a first order unfolding of a foliation defined by an integrable global twisted differential form $\omega\in H^0(\PP^n,\Omega^1_{\PP^n}(e))$ is just an integrable twisted differential form $\widetilde{\omega}\in H^0(\PP^n\times \mathcal{D},\Omega^1_{\PP^n\times\mathcal{D}|\CC}(e))$ that can be written as
 \[
  \widetilde{\omega} = \omega+ \varepsilon \eta  +h d\varepsilon\ .
  \]
  The integrability condition on $\widetilde{\omega}$ gives the equations
  \[
   \left\{
\begin{aligned}
 &\omega\wedge d\eta+d\omega\wedge \eta= 0\\
 &hd\omega=\omega\wedge (\eta-dh)
\end{aligned}
\right.
  \]
in an analog way to \cref{ecu0.2}. By the same argument of \cref{equnfcodq} they are equivalent to just the second equation
\[
   \left\{
\begin{aligned}
 hd\omega = \omega\wedge (\eta-dh)\ .
 \end{aligned}\right.
\]
Then, we can parameterize first order unfoldings as
\[
 U(\omega) = \{(h,\eta) \in H^0(\PP^n,\OO_{\PP^n}(e))\times H^0(\PP^n,\Omega^1_{\PP^n}(e)): hd\omega = \omega\wedge (\eta-dh)\ \}\big/ (0,\omega)\ ,
\]
where we are quotienting by $(0,\omega)$ to eliminate the trivial deformation induced by $\omega$, \emph{i.e.}, $\widetilde{\omega} = \omega + \varepsilon \omega = (1+\varepsilon) \ \omega$. Since $\varepsilon ^2=0$ we have that $1+\varepsilon$ is invertible, implying that $\omega$ and $(1+\varepsilon) \ \omega$ define the same foliation.

\

Let us denote as $(\CC^n,0)$ the infinitesimal neighbourhood around the origin of $\CC^n$, $\OO_{(\CC^n,0)}$ the ring of germs of analytic functions and $\Omega^1_{(\CC^n,0)}$ the germs of differential forms and consider $\omega_0\in \Omega^1_{(\CC^n,0)}$ a germ of an integrable differential form. Regarding the equations above we get all the same setting, \emph{i.e.},
\[
 U_{(\CC^n,0)}(\omega_0) = \{ (h,\eta)\in \OO_{(\CC^n,0)}\times \Omega^1_{(\CC^n,0)}: h d\omega_0=\omega_0\wedge(\eta-dh)\}\big/ (0,\omega_0)\ ,
\]
the only (and main) difference is that the image of the projection to the first coordinate
\[
\xymatrix@R=0pt{ U_{(\CC^n,0)}(\omega_0)  \ar[r]^{\pi_1} & \OO_{(\CC^n,0)}\\
(h,\eta) \ar@{|->}[r] & h}
\]
defines an ideal of the ring of germs of analytic functions $\OO_{(\CC^n,0)}$, $Im(\pi_1(\omega_0)) = I(\omega_0)$, instead of a finite vector space in the algebraic case. If we ask $\omega_0$ to have singular locus of codimension $\geq 2$, it can be easily seen that
\[
 I(\omega_0) \simeq U(\omega_0)\ .
\]
The ideal $I(\omega_0)$ defines the unfolding ideal on $(\CC^n,0)$ or, in our terms, the ideal defining the variety of persistent singularities. For more details on the local setting we refer to \cite{suwa-unfoldings_complex} and for the global algebraic setting to \cite{moli}.
\end{remark}

Let $\LL\xrightarrow{\omega}\Omega^q_X$ be an integrable morphism inducing a morphism $\mathcal{E}\to\Omega^1_X$. Composing $\omega$ with wedge product gives a morphism
 $\Omega^2_X\xrightarrow{\omega\wedge -} \LL^{-1}\otimes\Omega^{q+2}_X$. As $\omega$ is integrable the sheaf $\mathcal{E}\otimes\Omega^1_X$ is in the kernel of $\omega\wedge -$.
\begin{definition}
Following the notation above, we define the sheaf $H^2(\omega)$ as
 \[
  H^2(\omega):= \ker(\omega\wedge -) /\mathcal{E}\otimes\Omega^1_X.
 \]
\end{definition}

\begin{remark}
 The restriction of the de Rham differential to $\mathcal{E}$ gives a sheaf map $\mathcal{E}\to \Omega^2_X$ which is not $\OO_X$-linear but whose image is in $ \ker(\omega\wedge -) $ as $\omega$ is integrable.
 The projection of this map to $H^2(\omega)$ is however $\OO_X$-linear as $d g\varpi\cong g d\varpi \mod \mathcal{E}\otimes\Omega^1_X$ for every local section $\varpi$ of $\mathcal{E}$.
\end{remark}

 Let us denote $\LL\xrightarrow{\omega}\Omega^q_X$ to be an integrable morphism determining a subsheaf $\mathcal{E}\to\Omega^1_X$. Then we have the following definitions:

\begin{definition}\label{defII}
The subscheme of \emph{persistent singularities} of $\omega$ is  defined by the ideal sheaf $\II(\omega)$ to be the annihilator of $d(\mathcal{E})$ in $H^2(\omega)$. In other words the local sections of $\II(\omega)$ in an open set $U\subseteq X$ are given by
 \[
  \II(\omega)(U)=\{ h\in \OO_X(U)\ :\  \forall \varpi\in\mathcal{E}(U),\ hd\varpi=\sum_j \alpha_j\wedge\omega_j \},
 \]
for some local 1-forms $\alpha_j$ and 1-forms $\omega_j$ in $\mathcal{E}(U)$. We will denote it just as $\II$ if no confusion arises.

We also define $\per$ as the subvariety (subscheme) of $X$ associated to the ideal sheaf $\II(\omega)$, \emph{i.e.},
\[ \per = \OO_X/\II(\omega).\]
\end{definition}

For the case $q=1$ in \cref{equnfcodq}, we recall the following proposition from \cite[Proposition 3.6]{mmq2} which gives the name to the ideal of persistent singularities:
\begin{proposition}\label{prop}
 Let $\pp \in X$ be a point in $\sing(\omega)$, $\OO_{X,\pp}$ the local ring around $\pp$, and $X_\pp=\mathrm{Spec}(\OO_{X,\pp})$. Then $\pp$ is in the subscheme of persistent singularities if and only if for any first order unfolding $\widetilde{\omega}$  of $\omega$ in $X_\pp$, the point $(\pp, 0)\in X_\pp\times \mathcal{D}$ is a singular point of $\widetilde{\omega}$.
\end{proposition}

\begin{definition}\label{defKK}
 The subscheme of \emph{Kupka singularities} of $\omega$ is the one defined by the ideal sheaf  $\KK(\omega) = \ann(\{d\omega\})$ where $\{d\omega\}$ is the class of $d\omega\in \Omega^{q+1}_X\otimes\OO_{\sing(\omega)}\otimes \LL^{-1}$.  We will denote it just as $\KK$ if no confusion arises.

We also define as $\kup$ the subvariety (subscheme) of $X$ associated to the variety defined by $\KK(\omega)$, \emph{i.e.},
\[
\kup = \OO_X/\KK(\omega).
\]
Let $\pp\in\kup$. We will call $\pp$ a \emph{strictly Kupka point} if $d\omega|_\pp\neq 0$, meaning that the evaluation of $d\omega$ at the point $\pp$ is different from zero.
\end{definition}

\begin{remark}
Again, in the case of a codimension-1 foliation in the projective space  $\PP^n$ we can also define the ideal $\KK(\omega)$ as the quotient ideal defined as $(\JJ(\omega).\Omega^2_{\CC[x_0,\ldots,x_n]|\CC}:d\omega)$. Also, the reduced structure associated to $\KK(\omega)$ differs from the variety defined as the closure of the strictly Kupka points, as it is explained in \cite[p.~1034]{mmq}.
\end{remark}

We will need the following proposition, also from \cite[Proposition 4.9]{mmq2}, for later use.

\begin{proposition}\label{JJcIIcKK2}
 Given an integrable morphism $\LL\xrightarrow{\omega}\Omega^q_X$ we have the inclusions $\per\subset \sing(\omega)$ and $\kup\subset\sing(\omega)$.
 If moreover $\omega$ is locally decomposable (\emph{i.e.} if $\EE$ is locally free) then we have $\kup\subset\per\subset\sing(\omega)$. 
\end{proposition}

\section{Differential forms with logarithmics poles}\label{poles}

Along this section recall the definitions of $q$-logarithmic differential forms and of residue of a differential form due to Deligne. For that we will follow \cite[3. P\^oles logarithmiques]{deligne}.

\bigskip

We will refer to the set $\{1,\ldots,s\}$ by $[s]$ and by $|I|$ to the cardinality of the (sub) set $I\subset [s]$.

\begin{definition} Let us consider $D$ a simple simple normal crossing divisor and let $j:X\backslash D \to X$ be the inclusion. We define the \emph{sheaf of logarithmic differential forms along $D$} as the smallest sheaf $\Omega^{*}_{X}(\log D)$ of $j_*\Omega^{*}_{X\backslash D}$, stable by wedge products and such that $\frac{df}{f}$ is a local section of $\Omega^1_X(\log D)$ when $f$ is a local section of $j_*\OO^*_X$, meromorphic along $D$.
\end{definition}

By \cite[Proposition 3.2]{deligne} we know that $\Omega^1_X(\log D)$ is locally free and
\[
 \Omega^q_X(\log D) = \bigwedge^q \Omega^1_X(\log D)\ .
\]
Recall that given a locally free sheaf $\mathcal{E}$ and a global section given by a morphism $\OO_X\xrightarrow{s}\mathcal{E}$ the \emph{zero locus} of $s$ is the subscheme defined by the sheaf of ideals defined by the image of the dual morphism $\mathcal{E}^{\vee}\xrightarrow{s^\vee}\OO_X$.

\begin{notation}
 We denote the zero locus of a section $\omega \in H^0(X,\Omega^q_X(\log D))$ by $Z(\omega)$.
\end{notation}

Let us consider $D=\left(\bigcup_{i=1}^s D_i\right)$, where each $D_i$ is an irreducible component of the divisor $D$ and $\widetilde{D}^r=\bigsqcup_{\{I\subset [s],\ |I|=r\}} \bigcap_{i\in I} D_i$, for $r\geq 1$. We will ask the $D_i$ to be smooth. In the case $r=1$ we will skip the superindex in $\wt{D}^1$ and write just $\wt{D}$ instead.

\bigskip

\

In \cite[p.~77]{deligne}  there is the definition of \emph{r\'esidu de Poincar\'e} as a map
\[
 Res: Gr_n^W(\Omega^q_X(\log D)) \to \Omega^q_{\widetilde{D}^n}(\epsilon^n)[-n]\ ,
\]
where the $\epsilon^n$ stands for choosing an ordering on the subsets of $\widetilde{D}^n$, \emph{i.e.} on the $I\subset [s]$, as it is explained in \cite[3.4, p. 75]{deligne}. In the following we will fix an ordering and remove the $\epsilon^n$.

\

Taking $n=q$ we get a map
\[
 Res: Gr_q^W(\Omega^q_X(\log D)) \to \Omega^q_{\widetilde{D}^q}[-q]=\OO_{\widetilde{D}^q}\ ,
\]

\

Then we can define the \emph{residue of a $q$-differential form} which consists of the extension by linearity the following application:

\[
\xymatrix@R=0pt{\Omega^q_X(\log D) \ar[r]^-{\mathrm{Res}^q } &\OO_{\wt{D}^q}\\
\frac{df_{i_1}}{f_{i_1}}\wedge\dots\wedge \frac{df_{i_q}}{f_{i_q}} \ar@{|->}[r]& 1 \in \OO_{\wt{D}^q}\\
\save[] *{\omega\ }="s" \restore & \save[] +<7pt,0pt> *{\  0, \text{ if }\omega\in \Omega^q_X}="t" \restore & \ar@{|->}"s";"t"
}
\]
where $i_1<\ldots<i_q$. If $\omega\in H^0(X,\Omega^q_X(\log D))$ we have
\[
\mathrm{Res}^q(\omega)\in H^0(X,\OO_{\wt{D}^q})=\bigoplus_{\substack{I\subset [s]  \\ |I|=q,\  1\in I}} H^0(X,\OO_{\bigcap_{i\in I} D_{i}})
.
\]

We can also define a residue application $R: \Omega^q_X(\log D)\to \Omega^{q-1}_{\wt{D}}(\log D')$ in the following way: first consider $D'$ a divisor of $\wt{D}$ defined by
\[
D'=\bigsqcup_{i=1}^s \wt{D}|_{D_i} = \bigsqcup_{\substack{j\neq i\\i,j=1}}^s \left(D_i\bigcap D_j\right),
\]
then we define $R$ as the map obtained by extending linearly the following,
\begin{equation}\label{R}
  \xymatrix@R=0pt{ \Omega^q_X(\log D)\ar[r]^-{R} & \Omega^{q-1}_{\wt{D}}(\log D')\\
 \frac{df_{i}}{f_{i}}\wedge \eta \ar@{|->}[r]& \eta|_{D_i},
 }
\end{equation}
where we would like to recall that we have fixed an ordering for the subsets of $D'$.

\

Notice that $\eta|_{D_i}$ will have logarithmic poles in the intersections of $D_i$ with the other components of $D$.

\begin{remark}\label{rem1}
 Consider $D_1$ as an irreducible component of $D$ and $\eta$ a $q$-logarithmic form in $\Omega^q_X(\log D)$. The logarithmic $(q-1)$-form $\eta|_{D_1}$ has as a residue
\[
\mathrm{Res}^{q-1}(\eta|_{D_1}) \in \bigoplus_{\substack{I\subset [s]  \\ |I|=q,\ 1\in I}} H^0(X,\OO_{\bigcap_{i\in I} D_i}).
\]
Meaning that the residue of $\eta|_{D_1}$ consists of those $\lambda_i$ that are residues of $\eta$ and such that they correspond to the intersection of $q$-components in which $D_1$ is one of them. In particular, if all the $\lambda_i$ of $\eta$ are $\neq 0$, then the same is true for the residues of the differential $(q-1)$-form $R(\eta)$.
\end{remark}

\begin{proposition}\label{propSingRes}
 Keeping the notation of the residue maps as above, lets consider $\omega\in H^0(X,\Omega^q_X(\log D))$ and define $\eta$ as $\eta=R(\omega)\in H^0(X,\Omega^{q-1}_{\wt{D}}(\log D'))$. Then we have that
 \[
  Z(\omega)\cap D\subset Z(\eta) .
 \]
\end{proposition}

\begin{proof}
Let us take $\pp\in Z(\omega)\cap D$, such that $\pp\in D_1$. Locally around $\pp$ we can write
\[
\omega = \frac{df_1}{f_1}\wedge\wt{\eta} +\wt{\omega},
\]
where $\wt{\eta}$ is a $(q-1)$-differential form in $X$ such that $\wt{\eta}|_{D_1}=\eta$ and $\wt{\omega}$ is an holomorphic $q$-differential form. Being $\pp$ a singular point then
\[
\frac{df_1}{f_1}\wedge\wt{\eta}|_\pp +\wt{\omega}|_\pp=0,
\]
since $\wt{\omega}$ is holomorphic it is linearly independent with $\frac{df_1}{f_1}\wedge\wt{\eta}$ near $\pp$ implying that
\[
\frac{df_1}{f_1}\wedge\wt{\eta}|_\pp =\wt{\omega}|_\pp=0.
\]
Then $\wt{\eta}|_\pp=0$ because $\frac{df_1}{f_1}$ is a local generator of $\Omega^1_X(\log D)$. Then we conclude that $\eta|_\pp=0$ as we wanted to see.
\end{proof}

\section{Main results}\label{main}

Now we define our main object of study, which are $q$-logarithmic foliations on a smooth variety $X$. For that, we first define a 1-logarithmic foliation. Then we will state and prove our main results, which are \cref{teo2}, \cref{corofinal}.

\begin{definition}\label{defqform}
Let $\varpi\in H^0(X,\Omega^1_X(\log D))$ be a 1-logarithmic form, where $D$ is a simple normal crossing divisor locally defined as $D=(\prod_{i=1}^s f_i)$, written as
\[
 \varpi = \sum_{i=1}^s\lambda_i \frac{df_i}{f_i}, \text{ with $\lambda_i\in\CC$}.
\]
Then we define the \emph{1-logarithmic foliation} associated to $\varpi$ as
\[
\omega= \left(\prod_{i=1}^sf_i\right) \sum_{i=1}^s\lambda_i \frac{df_i}{f_i}= \sum_{i=1}^s\lambda_i F_{\widehat{i}}\,  df_i\ ,
\]
where $F_{\widehat{i}}\, =\prod_{\substack{j=1\\j\neq i}}^s f_j$. This defines a global section
\[
 F\varpi \in H^0(X,\Omega^1_X\otimes\LL)\ ,
\]
for some line bundle $\LL$, and $F$ locally given by $\prod_{i=1}^s f_i$.

\end{definition}

\begin{definition}\label{defqfol}
 Let $\varpi\in H^0(X,\Omega^q_X(\log D))$ be a locally decomposable $q$-\linebreak logarithmic differential form, where $D$ is a simple normal crossing divisor locally defined as $D=(\prod_{i=1}^s f_i)$, written as
\[
 \varpi = \sum_{\substack{I\subset [s],\ |I|=q}} \lambda_I \frac{df_I}{f_I}\ .
\]
Then we define the \emph{$q$-logarithmic foliation} associated to $\varpi$ as
 \begin{equation}\label{q-logfol}
    \omega = \left(\prod_{i=1}^sf_i\right) \left(\sum_{\substack{I\subset [s],\ |I|=q}} \lambda_I \frac{df_I}{f_I}\right) = \sum_{\substack{I\subset [s],\ |I|=q}} \lambda_I F_{\widehat{I}} df_I\ ,
 \end{equation}
where $F_{\widehat{I}}$ denotes the product $\prod_{\substack{j\in [s],\ j\notin I}} f_j$ and $df_I=\wedge_{i\in I} df_i$. This defines a global section
\[
 F\varpi\in H^0(X,\Omega^q_X\otimes\LL)\ ,
\]
for some line bundle $\LL$ and $F$ locally given by $\prod_{i=1}^s f_i$. We will denote this space as $Log(X,D,\LL, q)$.
\end{definition}

\begin{remark}\label{logfol1}
 If $\omega$ is a $q$-logarithmic foliation then it is indeed a foliation. Since it is locally decomposable, then it verifies condition c) of \cref{folcodimq}. Condition b) comes from the fact that the functions $f_i$ are tranversal. And condition a) can be seen because c) holds and $\omega$ is a closed $q$-differential form multiplied by a function $F$.
 
 Moreover, the smooth part of the divisor $D$ is invariant by the foliation $\omega$. Indeed, using \cref{q-logfol}, we can see that the restriction of $\omega$ to the irreducible component $D_i=(f_i=0)$ of $D$ in $\Omega^{q}_X|_{D_i}$ is of the form
 \[
 	\omega|_{D_i} = df_i|_{D_i}\wedge \eta
 \]
 For some form $\eta$ in $\Omega^{q-1}_X|_{D_i}$. This in turn implies that $df_i|_{D_i}$ is in $\EE(\omega)|_{D_i}$, which means that the distribution annihilated by $\omega$ is a subsheaf of the tangent sheaf to the divisor $D_i$ i.e.: the divisor $D_i$ is invariant by the foliation $\omega$.
\end{remark}

\begin{remark}\label{genericity}
 Following \cref{defqfol}, we will refer to a \emph{generic} $q$-logarithmic foliation as one where all the line bundles induced by the divisors $f_i$ are ample and the scalar coefficients $\lambda_I$ are required to be pairwise distinct. See \cite[\S4, 4.7]{egaIII}.
\end{remark}

In \cref{corofinal} we will see that under generic conditions the ideal of persistent singularities $\II(\omega)$ can be given by the following formula
\[
 \II(\omega) = \bigcap_{\substack{K\subset [s]\\|K|=q+1}} \sum_{i\in K} \mathcal{I}(f_i)\ ,
\]
where by $\mathcal{I}(f)$ we denote the ideal generated by $f$.
For that, we would like to state now the following technical lemma, which is a generalization of \cite[Corollary 1.6]{suwa-multiform}.

\begin{lemma}\label{lemma1}
Let us consider an open set $U\subset X$ and $f_1,\ldots, f_{s}\in H^0(U,\OO_X)$ sections of $X$ in $U$, for $s\in\NN$ such that $s> q$. Suppose that $ht(f_{i_1},\ldots,f_{i_{s-q}})=s-q$ if $i_1,\ldots,i_{s-q}\in\{1,\ldots,s\}$ are distinct, where `$ht$' stands for \emph{height}. Assume that $J$ runs through every subset of cardinality $q$ of $\{1,\ldots,s\}$, $q\geq 1$, then we have
 \[
  \sum_{\substack{J\subset [s], \ |J|=q}}\mathcal{I}\left(F_{\widehat{J}}\right) = \bigcap_{\substack{K\subset [s]\\ |K|=q+1}} \sum_{i\in K}\mathcal{I}\left( f_{i}\right) \ .
  \]
\end{lemma}

\begin{proof}To see the first inclusion
\[
\sum_{\substack{J\subset [s], \ |J|=q}}\left(F_{\widehat{J}}\right) \subset \bigcap_{\substack{K\subset [s]\\ |K|=q+1}} \sum_{i\in K}\mathcal{I}\left( f_{i}\right)
\]
we just notice that $F_{\widehat{J}}$ has $s-q$ factors. Now, since $s-q+q+1=s+1>s$ then for every subset $K\subset\{1,\ldots, s\}$ such that $|K|=q+1$ there exists a $k_0\in K$ such that $k_0\notin J$. This implies that $f_{k_0}|F_{\widehat{J}}$ and $f_{k_0}\in \sum_{i\in K} \mathcal{I}(f_i)$.

\

Since this happens for every $J\subset\{1,\ldots,s\}$, $|J|=q$, and $K\subset\{1,\ldots,s\}$, $|K|=q+1$, we have that $F_{\widehat{J}}\in\sum_{i\in K} \mathcal{I}(f_i)$. Then
\[
 \sum_{\substack{J\subset [s], \ |J|=q}}F_{\widehat{J}}\subset \bigcap_{\substack{K\subset [s]\\ |K|=q+1}} \sum_{i\in K} \mathcal{I}(f_i)
\]
as we wanted to see.

\

For the other inclusion we will use the transversality of the $\{f_1,\ldots,f_s\}$ in the following sense: we can compute the intersection of $\mathcal{I}(f_i)\cap \mathcal{I}(f_j) = \mathcal{I}(f_i.f_j)$ if $i\neq j$. Then we would like to see which is the worst choice of polynomials in the intersection $\bigcap_{\{K\subset [s],\  |K|=q+1\}} \sum_{i\in K} \mathcal{I}(f_i)$ in the sense that we would get the least amount of polynomials in the product. Let us suppose that we choose $f_1$ as many times as possible, then we take $f_2$ as many times as possible and so on. For $f_1$ we have $\binom{s-1}{q}$ possible choices, meaning that there are $\binom{s-1}{q}$ ideals with $f_1$ in the intersection; for $f_2$ we have $\binom{s-2}{q}$ possible choices, and it keeps going on like this. By the following formula
\[
 \sum_{i=1}^{s-q} \binom{s-i}{q} = \binom{s}{q+1}
\]
we are seeing that we can take, at least, $s-q$ polynomials to fill the $\binom{s}{q+1}$ possible choices of each one of the ideals in the intersection. Then, if $|J|=q$, we have that $F_{\widehat{J}}\in \bigcap_{\{K\subset [s],\  |K|=q+1\}}  \sum_{i\in K} \mathcal{I}(f_i)$ concluding our proof.
\end{proof}

From now on let us fix $X$ to be a smooth variety, $D=\sum_{i=1}^s D_i= \left(\prod_{i=1}^s f_i\right)$ a simple normal crossing divisor such that its components $D_i=(f_i)$ are irreducible, smooth, and ample.

We will also consider $\varpi\in H^0(X,\Omega^q_X(\log D))$ a $q$-logarithmic differential form and $\omega =F\cdot \varpi \in H^0(X,\Omega^q_X\otimes \LL)$ the associated logarithmic $q$-differential foliation, where $F=\prod_{i=1}^s f_i$, as in \cref{defqfol}.

\begin{proposition}\label{propDimZomega}
Let $\varpi\in H^0(X,\Omega^q_X(\log D))$ be a $q$-logarithmic differential form such that $q\leq n-2$, every residue is not zero (\emph{i.e.}: such that every $\lambda_I$ is $\neq 0$ for every $I$), and such that $\varpi$ is locally decomposable as a product of $q$ logarithmic $1$-forms outside its zero locus. Then the zero locus of $\varpi$, $Z(\varpi)$, has dimension $\dim(Z(\varpi))= q-1$ if $q\leq n-2$ or $Z(\varpi)=\emptyset$.
\end{proposition}
\begin{proof}
The case $q=1$ is the main result of \cite{fmi}. Let $q\geq 2$ and suppose $\dim(Z(\varpi))\geq q$ (in particular we are assuming that $Z(\varpi)\neq \emptyset$). The ampleness of the components $D_i$ implies that the intersection of the zeroes of $\varpi$, with $D$ has dimension greater or equal to $q-1$. By \cref{propSingRes} the intersection of $Z(\varpi)$ with $D$ is inside the zero locus of the form $R(\varpi)$. $R(\varpi)$ is a $(q-1)$-logarithmic form on a complete intersection, the residues of $R(\varpi)$ are a subset of the residues of $\varpi$, so they are all non-null. Then, by induction, the set of zeros of $R(\varpi)$ must be of $\dim(Z(R(\varpi)))= q-2$, therefore $\dim(Z(\varpi))\leq q-1$. Given that $\EE$ is a locally free sheaf locally decomposable of rank $n$ and $s$ is a section of $\bigwedge^q \EE$, since $q\neq n-1$, then the expected dimension of $Z(s)$ is $q-1$, so if $Z(s)\neq \emptyset$ we have $\dim(Z(s))\geq q-1$. Therefore $\dim(Z(\varpi))=q-1$.
\end{proof}

\begin{remark} The expected dimension of the singular locus of a generic section of a vector bundle of rank $n$ is zero. In the above proposition, this would be the case if $q=n-1 $, in which case the condition corresponds to a foliation by curves. For further details on this topic, we refer to the article by Corrêa and D. da Silva Machado \cite{correa_machado}.
\end{remark}

\begin{proposition}\label{propKupka}
	Let $\omega\in H^0(X,\Omega^q_X\otimes \LL)$ be a $q$-logarithmic foliation, as in \cref{q-logfol}, and suppose that $\lambda_I\neq\lambda_J$ if $I\neq J$. The Kupka scheme of $\omega$ is given by
	\[
		\kup =\bigcup_{K\subset [s],\ |K|=q+1}\left(\bigcap_{i\in K} D_i \right).
	\]
\end{proposition}
\begin{proof}
	We will proceed by induction on $q$. Suppose $q=1$, then \cite{fmi} implies $\sing(\omega)$ consists of isolated points and a codimension-$2$ subvariety given by the intersections of the components of $D$. If $\pp\in D_i\cap D_j$ is a regular point of $D_i\cap D_j$ then locally around $\pp$ we have $d\omega|_\pp = (\lambda_i-\lambda_j) g df_i\wedge df_j$ where $g$ is a unit of $\OO_{D_i\cap D_j,\pp}$. Therefore $\pp$ is a (strictly) Kupka  point.
	The case $q=1$ follows form the fact that, if $\sing(\omega)$ is reduced, then $\kup$ is the closure of the set of strictly Kupka points \emph{i.e.}, points $\pp\in\sing(\omega)$ such that $d\omega|_\pp\neq 0$.

	If $q>1$ then from the local expression of $\omega$ follows that  $\bigcap_{\{i_1,\dots, i_{q+1}\}} D_{i_j}\subseteq \sing(\omega)$ for any intersection of $q+1$ components of $D$. Let $\pp$ be a regular point of $\bigcup_{\{K\subset [s], |K|=q+1\}}\left(\bigcap_{i\in K} D_i \right)$ and suppose $d\omega|_\pp=0$. The logarithmic $(q-1)$-form $\eta:=R(\varpi)|_{D_j}$ defines a logarithmic foliation on $D_j$ given by the twisted form $F_{\widehat{j}} \, \eta$ whose Kupka scheme consists of the $q+1$-wise intersections of components of $D$ containing $D_j$. From the definition of $R(\varpi)$, \cref{R}, and \cref{propSingRes} it follows that if $d\omega|_\pp=0$ then $d(F_{\widehat{j}} \, \eta)|_\pp=0$ which contradicts the inductive hypothesis on $q$. So $\pp$ is a Kupka point of $\omega$ and the result follows.
\end{proof}

\

For the next Proposition we will need the following technical lemmas.

\

The next \cref{lemma11} is, somehow, a generalization of a result from Suwa, \cite[Theorem 1.11]{suwa-multiform}, that says that a universal unfolding for \linebreak$\omega= {f_1\ldots f_s\sum_{k=1}^s \lambda_k \frac{df_k}{f_k}}$ is of the form
\[
\widetilde{\omega} =\widetilde{f}_1\ldots \widetilde{f}_s\sum_{k=1}^s \lambda_k \frac{d\widetilde{f}_k}{\widetilde{f}_k},
\]
where $\widetilde{f}_k(x,t)$ is a function germ such that $\widetilde{f}_k(x,0)=f_k(x)$.

In case of a first order unfolding, \emph{i.e.} in $X\times \mathcal{D}$, where $\mathcal{D}=\mathrm{Spec}(k[\varepsilon]/(\varepsilon^2))$, we will have that $\widetilde{\omega} = \omega + \varepsilon \theta+ h d\varepsilon$, for some 1-logarithmic foliaton $\theta$ and some function germ $h$. Notice that the 1-logarithmic foliation $\theta$ is just (the addition of) $\omega$ with one of the $f_k$'s changed by another function germ, $g$, and the function $h$ is the product of the $f_\ell$'s, with $\ell\neq k$,  and $g$, and the  scalar $\lambda_k$. Without loss of generality we can suppose that $k=1$ and then we can write an unfolding $\widetilde{\omega}$ as

\begin{align}\label{eculemma2}
	\widetilde{\omega} &= F_{\widehat{1}}(f_1+\varepsilon g) \left(\lambda_1 \frac{d(f_1+\varepsilon g)}{f_1+\varepsilon g} + \sum_{\ell=2}^s \lambda_\ell \frac{df_\ell}{f_\ell}\right) \ .
\end{align}
Also notice that, by \cref{eculemma2} $\widetilde{\omega}\in H^0(X\times\mathcal{D},\Omega^1_{X\times\mathcal{D}|\CC}\otimes \widetilde{\LL})$ is also a $1$-logarithmic foliation in $X\times \mathcal{D}$.

\

Lets consider now $\varpi \in H^0(X,\Omega^q_X(\log D))$ to be a locally decomposable $q$-logarithmic differential form locally written as $\varpi = \omega_1\wedge\ldots\wedge\omega_q$, for \linebreak$\omega_i\in H^0(X,\Omega^1_X(\log D))$, $i=1,\ldots,q$, and  $\omega$ its associated logarithmic foliation. Let us write $\omega_i$ as
\[
 \omega_i = \lambda_1^i \ \frac{df_1}{f_1} + \omega_i^+ ,\ \ \  i=1,\ldots,q\ .
\]
Then $\omega$ can be written as
\begin{equation}\label{ecu1}
  \omega = F\left(\left( \sum_{i=1}^q (-1)^{i+1}\lambda_1^i \ \frac{df_1}{f_1}\wedge \omega_{\widehat{i}}^+ \right)+ \omega^+\right)= F\ \omega_1\wedge\ldots\wedge\omega_q\ ,
\end{equation}
where $\omega^+ = \omega_1^+\wedge\ldots\wedge\omega_q^+$ and we recall from \cref{ecu0.2} that $\omega_{\widehat{i}}^+ = \omega_1^+\wedge\cdots\wedge\widehat{\omega_i^+}\wedge\cdots\wedge\omega_q^+$.

\begin{lemma}\label{lemma11}
Following the notation above, let $\varpi \in H^0(X,\Omega^q_X(\log D))$ be a locally decomposable $q$-logarithmic differential form and  $\omega$ its associated logarithmic foliation. 

Then, for some locally defined function $g$,  there is a first order unfolding of $\omega$ of the following form
\begin{align}\label{eculemma}
     \widetilde{\omega} &=F_{\widehat{1}}(f_1+\varepsilon g) \left( \left(\sum_{i=1}^q (-1)^{i+1}\lambda_1^i \ \frac{d(f_1+\varepsilon g)}{f_1+\varepsilon g}\wedge \omega_{\widehat{i}}^+ \right)+\omega^+\right)= \\
     &=F_{\widehat{1}}(f_1+\varepsilon g)\ \left[\left(\lambda_1^1\  \frac{d(f_1+\varepsilon g)}{f_1+\varepsilon g} +\omega_1^+\right)\wedge \ldots\wedge \left(\lambda_1^q \ \frac{d(f_1+\varepsilon g)}{f_1+\varepsilon g} +\omega_q^+\right)\right]\nonumber
\end{align}
\end{lemma}
\begin{proof}

The proof can be established by directly computing \cref{folcodimq}, see \cref{logfol1}.
\end{proof}

\begin{lemma}\label{lemma2}
 Let $\varpi \in H^0(X,\Omega^q_X(\log D))$ be a locally decomposable $q$-logarithmic differential and $D_1=(f_1=0)$ one of the smooth irreducible components of the divisor $D$. Then, we can write $\varpi|_{D_1}$ in $\Omega^q_X(\log D)|_{D_1}$ as
 \[
  \varpi|_{D_1}= \frac{df_1}{f_1}\wedge\omega_2\wedge\ldots\wedge\omega_q\ .
 \]
\end{lemma}
\begin{proof}
 By \cref{logfol1} we have that $D_1$ is invariant by the foliation $\omega$ associated with $\varpi$, so $df_1|_{D_1}\wedge\omega|_{D_1}=0$ in $\Omega^{q+1}_X|_{D_1}$. Then, as $\omega=\prod_i f_i \cdot \varpi$, we have that, if $\omega|_{D_1}=df_1|_{D_1}\wedge \omega'$ then 
 \[ 
    \varpi|_{D_1} = \frac{df_1}{f_1}|_{D_1}\wedge \omega',
 \]
in $\Omega^{q+1}_X(\log D)|_{D_1}$, as $\varpi|_{D_1}$ is locally decomposable and $\frac{df_1}{f_1}|_{D_1}$ is a local generator of $\Omega^{q+1}_X(\log D)|_{D_1}$, the above equation implies 
\[
	\prod_{i\neq 1}f_i\cdot \varpi|_{D_1}= \left.\frac{df_1}{f_1}\right|_{D_1}\wedge\omega_2\wedge\cdots\wedge\omega_q\ .
\]
\end{proof}

As a corollary of the previous lemma we have the following statement:

\begin{corollary}
 $\varpi \in H^0(X,\Omega^q_X(\log D))$ be a locally decomposable $q$-logarithmic differential. Then we have that the residue operation defined in \cref{R} as
 \begin{equation*}
  \xymatrix@R=0pt{ \Omega^q_X(\log D)\ar[r]^-{R} & \Omega^{q-1}_{\wt{D}}(\log D')\\
 \frac{df_{i}}{f_{i}}\wedge \eta \ar@{|->}[r]& \eta|_{D_i},
 }
\end{equation*}
takes locally decomposable global sections of $\Omega^q_X(\log D)$ and maps them to locally decomposable global sections of $\Omega^{q-1}_{D_i}(\log D')$ for each irreducible component $D_i$ of $D$, so we get map
 \begin{equation*}
  \xymatrix@R=0pt{ Log(X,D,\LL,q) \ar[r]^-{R} & Log(D_i,D',\LL,q-1).
 }
\end{equation*}
\end{corollary}
\begin{proof}
 The proof is immediate using \cref{lemma2} and \cref{logfol1}.
\end{proof}

\begin{remark}\label{lemma3}
 Notice that taking the residue along $\{f_i=0\}$ and taking an unfolding of the type of \cref{lemma11}, by changing $f_j$ to $f_j+\varepsilon g$, with $j\neq i$, commute.
\end{remark}

\begin{proposition}
	Following the notation and assumptions above, if a point $\pp\notin \kup$ then there is a local first order unfolding $\widetilde{\omega}$, of the type of \cref{lemma11}, of $\omega$ around $\pp$ such that $\widetilde{\omega}|_\pp\neq 0$.
\end{proposition}
\begin{proof}
	Let us first assume that $\pp\notin D$, then the foliation is also defined locally by $\varpi$ which is the product of $\omega$ by a non-vanishing local function and also is (locally) decomposable as $\varpi=\omega_1\wedge\dots\wedge\omega_q$. To prove the existence of an unfolding $\widetilde{\omega}$ as in the proposition we need to show that there are $1$-forms $\alpha_{ij}$ such that $\omega_i=\sum_j d\omega_j\wedge\alpha_{ij}$. 
	In order to prove this we will use the \emph{generalized de Rham lemma} of \cite{saito} which states that, if the depth   of the singular ideal of the decomposable form $\varpi$ is greater than $3$ then for every $2$-form $\beta$ such that $\beta\wedge\varpi=0$ there are $1$-forms $\alpha_1,\dots,\alpha_q$  such that $\beta=\sum_i\omega_i\wedge\alpha_i$.
	Now, being a global logarithmic form, $d\varpi=0$. Then, for any $1$-form $\eta$ such that $\eta\wedge\varpi=0$, we have 
	\[
		d\eta\wedge\varpi=d(\eta\wedge\varpi)=0.
	\]
	If $q<n-1$, that is, if $\omega$ is \emph{not} a foliation by curves, then we have by \cref{propDimZomega} that $\codim(\sing(\omega))\geq 3$, which, as $X$ is a regular variety, in particular locally Cohen-Macaulay, implies that $\mathrm{depth}(\JJ)\geq 3$. Then we can apply the {generalized de Rham lemma} which implies there are $1$-forms $\alpha_1,\dots,\alpha_q$ such that
	\[
		d\eta = \sum_i \omega_i\wedge\alpha_i.
	\]
	So $1\in \II(\varpi)=\II(\omega)$, the definition ideal of $\per$, then $\pp$	is never in $\per$. Then, by \cite[Proposition 4.1]{mmq2} there exists an $\widetilde{\omega}$ unfolding of $\omega$ such that $\widetilde{\omega}|_{\mathfrak{p}}\neq 0$.

For the case when $\mathfrak{p}\in D$ we can suppose that $\mathfrak{p}\in \{f_1=0\}$.
We will employ an inductive approach. The case where $q=1$ is given by \cref{eculemma2} using \cite[Theorem 1.11]{suwa-multiform} and \cite[Proposition 3.6, Theorem 3.14]{mmq2}.

Consider a $q$-logarithmic foliation given by $\omega$ and take the residue at $\{f_1=0\}$, $R(\frac{\omega}{F})=\eta$. Now we can consider an unfolding of $\eta$ following \cref{lemma11} by changing $f_2$ to $f_2+\varepsilon g$, let us call this unfolding by $\widetilde{\eta}$. Using the inductive hypotheses we have that $\widetilde{\eta}(\mathfrak{p})\neq 0$.

By \cref{lemma3},  $\widetilde{\eta}$ is the residue along $\{f_1=0\}$ of an unfolding of $\omega$ given by changing $f_2$ to $f_2+\varepsilon g$.
By virtue of the linearity of the residue operation, if the residue is nonzero at $\mathfrak{p}$, then the differential form must also be nonzero at $\mathfrak{p}$, thereby completing the proof. \end{proof}

\begin{corollary}\label{coro1}
We have the inclusion of closed subvarieties
\[
(\per)_{\mathrm{red}}\subseteq (\kup)_{\mathrm{red}}\ .
\]
\end{corollary}
\begin{proof}
	Around a point $\pp\notin D$, by the same argument of the above proof we have that $\pp$ is never in $\per$. If $\pp\in D$ then the above proposition together with \cite[Proposition 4.1 ]{mmq2} is showing that $\pp\notin \kup$ implies $\pp\notin \per$, from which the corollary follows.
\end{proof}

\begin{proposition}\label{prop1}
	Let $\pp\in\kup$ be a strictly Kupka point. Suppose that, when representing $\mathrm{Res}^q(\omega)$ as a vector in
	\[
	\mathrm{Res}^{q}(\omega) \in \bigoplus_{ I\subset [s],\ |I|=q} H^0(X,\OO_{\bigcap_{i\in I} D_i})\simeq \CC^{|\{ I\subset [s],\ |I|=q\}|}\ ,
	\]
	 there are two coordinates with different coefficients $\lambda_I\neq\lambda_J$.
	Then locally around $\pp$ we have $\II_\pp\subset \KK_\pp$.
\end{proposition}
\begin{proof}
	Being $\pp$ a Kupka point, following \cite[Proposition 1.3.1]{deMed1}, we can find local coordinates $(x_1,\dots, x_n)$ in which $\omega$ is written completely in terms of $(x_1,\dots, x_{q+1})$. In other words there is a local submersion $\pi:U\to \CC^{q+1}$ and a local form $\alpha\in \Omega^q_{\CC^{q+1},0}$ such that $\omega|_U=\pi^*\alpha$. We can take $U$ small enough in such a way that $0\in\CC^{q+1}$ is the only singularity of $\alpha$. Moreover, $\alpha$ also defines a logarithmic foliation in $\CC^{q+1}$ and its isolated singularity is the normal crossing intersection of $q+1$ of its invariant divisors.
	So we have local coordinates around $\pp$ such that the logarithmic form $\frac{\omega}{F}$ is written as
	\[
		\frac{\omega}{F} = \sum_{i=1}^{q+1} f_i \bigwedge_{j\neq i} \frac{dx_j}{x_j} + \xi_1.
	\]
	Where in this coordinates $F=x_1\cdots x_{q+1}$ and where $\xi_1\in W^{q-1}\Omega^q_X(\log D)$, the space of germs of logarithmic $q$-forms of weight at most $q-1$ as defined in \cite[3.5]{deligne}.
	Therefore we have in this coordinates the expression
	\[
		\omega = \sum_{i=1}^{q+1} x_i f_i \bigwedge_{j\neq i}dx_j +\xi_2.
	\]
	In this case $\xi_2 = F\cdot \xi_1$ is the germ of a holomorphic $q$-form which is actually in $(x_1,\dots, x_{q+1})^2\Omega_{X,\pp}^q= K^2\cdot \Omega_{X,p}^q$.

	Now lets take $\Xi_{i,k} = \bigwedge_{j\neq i,k}\frac{\partial}{\partial x_j}\in \bigwedge^{q-1}T_{X,\pp}$ and $\varpi_{i,k} = i_{\Xi_{i,k}} \omega$, then
	\[
		\varpi_{i,k} = x_if_i dx_k - x_k f_k dx_i + i_{\Xi_{i,k}}\xi_2.
	\]
	Notice that the $\varpi_{i,k}$ with $1\leq i<k\leq q+1$ generate the sheaf $\EE(\omega)$, as $i_{\frac{\partial}{\partial x_j}} \omega = 0$ for $q+2\leq j\leq n$.
	In particular every 1-form in $\EE$ is in $K\cdot \Omega^1_{X,0}$.
	We also have
	\[
		d\varpi_{i,k} = (f_i-f_k) dx_i\wedge dx_k + \xi_3,
	\]
	where $\xi_3 = d\xi_2$ is the germ of a $2$-form in $K\cdot\Omega^2_{X,0}$.

	Now suppose $h\in \II_\pp$, we want to show $h=0$ in $\OO_{X,\pp}/K_\pp$. As $h\in \II_\pp$ we have that, for every $\varpi\in\EE$,
	\[
		hd\varpi \in K\cdot \Omega^2_{X,\pp}.
	\]
	In particular we may take  $\varpi_{i,k}$ and get
	\[
		0\equiv hd\varpi_{i,k} \equiv h (f_i(0)-f_k(0)) dx_i\wedge dx_k \mod K.
	\]
	If we have two different residues in $\omega$ then we may take $i,k$ such that $f_i(0)-f_k(0)\neq 0$ (because the residues are exactly $f_i(0)$ for $i=1,\dots, q+1$). Hence we must have $h\equiv 0 \mod K$.
\end{proof}

By using the two propositions above we can conclude the following.

\begin{theorem}\label{teo1}
 Let $\omega$ be a codimension-$q$, locally decomposable, logarithmic foliation on $X$ such that $\kup$ is reduced and the singular locus be the disjoint union of $\kup$ and a variety $\HH$, \emph{i.e.}, $\sing(\omega) =\kup \sqcup \HH$. Then we have
 \[
  \kup= \per .
 \]
\end{theorem}

\begin{proof}
 By \cref{coro1} we have that $(\per)_{red}\subset(\kup)_{red}$, then by our hypothesis we have that $(\per)_{red}\subset \kup$. And by \cref{prop1} we know that for every $\pp\in \kup$ a strictly Kupka point we have $\II_\pp\subset\KK_\pp$, meaning that $\kup_\pp\subset \per_\pp$.

 Since $\kup$ is reduced then the set of strictly Kupka points is dense in $\kup$, implying that
 \[
  \kup\subset\per\ .
 \]
By \cref{JJcIIcKK2} we have that $\per\subset \sing(\omega)$ then $\kup\subset\per\subset\sing(\omega)$. And since $\sing(\omega)=\kup\sqcup\HH$ we have that
\[
 \kup_\pp\subset\per_\pp\subset \sing(\omega)_\pp=\kup_\pp\ , \forall \pp\in Supp(\kup)
\]
implying that $\KK(\omega)=\II(\omega)$.
\end{proof}

As a corollary of the theorem above we get the following decomposition of $\sing(\omega)$ for $\omega$ a $q$-logarithmic foliation.

\begin{theorem}\label{teo2}
Let $\omega\in H^0(X,\Omega^q_X\otimes \LL)$ be a generic, locally decomposable, $q$-logarithmic foliation with $q\leq n-2$. Then we can decompose its singular locus $\sing(\omega)$ as the disjoint union of
\[
 \sing(\omega) = \kup\sqcup \HH,
\]
where $\kup$ is the Kupka variety of $\varpi$ of codim $q+1$ and $\HH$ is a variety of dimension $q-1$ or $\HH=\emptyset$. Even more so, we have that $\kup=\per$.
\end{theorem}
\begin{proof}
Since $\omega$ is a $q$-logarithmic foliation, we know that its singular locus is made of $\kup$ of codimension-$(q+1)$ and  a variety $\HH$ of dimension $(q-1)$ or $\HH=\emptyset$. Since $\omega$ is generic, we can assume that $\kup\cap \HH=\emptyset$.

Finally, we can apply  \cref{teo1} and conclude $\kup=\per$ as we wanted to see.
\end{proof}

Let $X$ be a smooth projective variety with an ample line bundle $\LL$. We will denote $\bigoplus_{n\in \ZZ}H^0(X,\FF\otimes \LL^n)$ by $H^0_*(X,\FF)$.

\begin{corollary}\label{corofinal}
	Let $X$ a smooth projective variety with an ample line bundle and $\omega$ a $q$-logarithmic foliation defined by a simple normal crossing divisor $D=(\prod_{i=1}^s f_i)$ as
	\[
        \omega = F\left(\sum_{I\subset [s],\ |I|=q} \lambda_I \frac{df_I}{f_I}\right) = \sum_{I\subset [s],\ |I|=q} \lambda_I F_{\widehat{I}} df_I= F\varpi\in H^0(X,\Omega^q_X\otimes\LL)
    \]
	satisfying the hypotheses of \cref{teo2}, then we have
	\[
		 H^0_*(X,\II(\omega)) = H^0_*(X,\KK(\omega)) =\sum_{J\subset [s],\ |J|=q}\mathcal{I}\left(F_{\widehat{J}}\right) = \bigcap_{\substack{K\subset [s]\\ |K|=q+1}} \sum_{i\in K}\mathcal{I}\left( f_{i}\right).
	\]
\end{corollary}

\begin{proof}
For $\omega$ of the form given in the statement the Kupka scheme is, by \cref{propKupka}, given by the ideal $\bigcap_{\substack{\{K\subset [s], |K|=q+1\}}} \sum_{i\in K}\mathcal{I}\left( f_{i}\right)$. Then applying \cref{teo2} and \cref{lemma1} we are done.
\end{proof}

%

\

\

\begin{tabular}{l l}
Ariel Molinuevo$^*$  &\textsf{arielmolinuevo@gmail.com}\\
Federico Quallbrunn$^\dag$  &\textsf{fquallb@dm.uba.ar}\\
\end{tabular}

\

\noindent
\small{
\begin{tabular}{l l}
& \\
$^*$Instituto de Matem\'atica & $^\dag$ Departamento de Matem\'atica\\
Universidade Federal do Rio de Janeiro & Universidad CAECE\\
Av. Athos da Silveira Ramos 149 & Av. de Mayo 866 \\
Bloco C, Centro de Tecnologia, UFRJ & CP C1084AAQ{ Buenos Aires} \\
Cidade Universitária, Ilha do Fund\~ao & {Argentina} \\
CEP. 21945-970 Rio de Janeiro - RJ & \\
Brasil & \\

\end{tabular}

}

\end{document}